\documentclass{amsart}
\usepackage{hyperref}

\newcommand*{\mailto}[1]{\href{mailto:#1}{\nolinkurl{#1}}}
\newcommand{\arxiv}[1]{\href{http://arxiv.org/abs/#1}{arXiv:#1}}
\newtheorem{theorem}{Theorem}[section]

\newtheorem{lemma}[theorem]{Lemma}

\newcommand{\R}{{\mathbb R}}
\newcommand{\N}{{\mathbb N}}

\newcommand{\C}{{\mathbb C}}

\newcommand{\be}{\begin{equation}}
\newcommand{\ee}{\end{equation}}

\newcommand{\ol}{\overline}

\newcommand{\T}{\mathbb{T}}
\DeclareMathOperator{\supp}{supp}

\newcommand{\eps}{\varepsilon}

\newcommand{\sig}{\sigma}

\numberwithin{equation}{section}

\begin{document}

\title{On the connection between the Hilger and Radon--Nikodym derivatives}

\author[J.\ Eckhardt]{Jonathan Eckhardt}
\address{Faculty of Mathematics\\ University of Vienna\\
Nordbergstrasse 15\\ 1090 Wien\\ Austria}
\email{\mailto{Jonathan.Eckhardt@univie.ac.at}}

\author[G.\ Teschl]{Gerald Teschl}
\address{Faculty of Mathematics\\ University of Vienna\\
Nordbergstrasse 15\\ 1090 Wien\\ Austria\\ and International
Erwin Schr\"odinger
Institute for Mathematical Physics\\ Boltzmanngasse 9\\ 1090 Wien\\ Austria}
\email{\mailto{Gerald.Teschl@univie.ac.at}}
\urladdr{\url{http://www.mat.univie.ac.at/~gerald/}}

\thanks{J. Math. Anal. Appl. {\bf 385}, 1184--1189 (2012)}
\thanks{{\it Research supported by the Austrian Science Fund (FWF) under Grant No.\ Y330}}


\begin{abstract}
We show that the Hilger derivative on time scales is a special case of the Radon--Nikodym derivative
with respect to the natural measure associated with every time scale.
Moreover, we show that the concept of delta absolute continuity agrees with the one from measure
theory in this context.
\end{abstract}

\maketitle

\section{Introduction}
\label{sec:int}

Time scale calculus was introduced by Hilger in 1988 as a means of unifying differential and difference
calculus. Since then this approach has had an enormous impact and developed into a new field of
mathematics (see e.g.\ \cite{bope}, \cite{bope2} and the references therein).
However, the aim to unify discrete and continuous calculus is of course much older and
goes back at least to the introduction of the Riemann--Stieltjes integral, which unifies sums and integrals,
by Stieltjes in 1894. Of course these ideas have also been used to unify differential and difference
equations and we refer to the seminal work of Atkinson \cite{at} or the book by Mingarelli \cite{ming}.
The inverse operation to the Lebesgue--Stieltjes integral is the Radon--Nikodym
derivative and it is of course natural to ask in what sense this old approach is related to the new
time scale calculus. Interestingly this question has not attained much attention and is still not fully
answered to the best of our knowledge. It is the aim of the present paper to fill this gap by showing that
the Hilger derivative equals the Radon--Nikodym derivative with respect to the measure which is
naturally associated with every time scale. It can be defined in several equivalent ways, for example
via its distribution function, which is just the forward shift function (cf.\ \eqref{borelm}), or
as the image of Lebesgue measure under the backward shift function (cf.\ \eqref{defimgmeas}).
This measure was first introduced by Guseinov in \cite{gus}
and it was shown by Bohner and Guseinov in Chapter~5 of \cite{bope2} that the delta integral
on time scales is a special case of the Lebesgue--Stieltjes integral associated with this measure (see
also \cite{cabviv2}, \cite{deuf}, \cite{mpt} for further results in this direction).

Moreover, Cabada and Vivero \cite{cabviv} introduced the concept of absolutely continuous functions
on time scales and proved a corresponding fundamental theorem of calculus. Again the natural question
arises, in what sense this new concept is related to the usual concept of absolute continuity with
respect to the natural measure associated with the time scale. Of course this is also related to the
concept of weak derivatives introduced by Agarwal, Otero-Espinar, Perera, and Vivero
\cite{agotpevi} (see also the alternative approach by Davidson and Rynne \cite{dary}, \cite{ry} via
completion of continuous functions).

Finally, our result also generalizes the work of Chyan and Fryszkowski \cite{chfr} who showed that
every increasing function on a time scale has a right derivative almost everywhere.

\section{The Hilger derivative as a Radon--Nikodym derivative}
\label{sec:hd}

To set the stage we recall a few definitions and facts from time scale calculus \cite{bope}, \cite{bope2}.
Let $\T$ be a time scale, that is, a nonempty closed subset of $\R$. We define the forward and backward shifts on $\R$ via
\be
\sig(t) = \begin{cases}
\inf \{ s \in\T \,|\, t < s \}, & t < \sup\T,\\
\sup\T, & t \ge \sup\T,
\end{cases}
\quad
\rho(t) = \begin{cases}
\sup \{ s \in\T \,|\, t > s \}, & t > \inf\T,\\
\inf\T, & t \le \inf\T,
\end{cases}
\ee
in the usual way. Note that $\sig$ is nondecreasing right continuous and $\rho$ is
nondecreasing left continuous. The quantity
\be
\mu(t) = \sig(t) -t, \quad t\in\T
\ee
is known as the graininess.
A point $t\in\T$ is called right scattered if $\sig(t)>t$ and left scattered if $\rho(t) < t$. Since a nondecreasing function can have
at most countably many discontinuities there are only countably many right or left scattered points.

Associated with $\T$ is a unique Borel measure which is defined via its distribution function
$\sig$ (this procedure is standard and we refer to, e.g., \cite[Sect.~A.1]{tschroe} for a brief
and concise account). For notational simplicity we denote this measure by the same letter $\sig$ and hence have
\be \label{borelm}
\sig(A) = \begin{cases}
\sig_+(b) - \sig_+(a), & A = (a,b],\\
\sig_+(b) - \sig_-(a), & A = [a,b],\\
\sig_-(b) - \sig_+(a), & A = (a,b),\\
\sig_-(b) - \sig_-(a), & A = [a,b).
\end{cases}
\ee
Here we use the short-hand notation
\be
f_\pm(t) = \lim_{\eps\downarrow 0} f(t\pm\eps)
\ee
for functions $f:\R \to \C$ which are locally of bounded variation (such that the limits always exist). Note that since
$\sig_-(t) = t$ for $t\in\T$ we have
\be
\sig(\{t\}) = \mu(t), \qquad t\in\T.
\ee
The topological support of $\sig$
is given by
\be
\supp(\sig)=\T^\kappa, \qquad
\T^\kappa =\ol{\T \backslash \{\sup\T\}}.
\ee
Note that $\T^\kappa=\T$ if $\T$ does not have a left scattered
maximum and otherwise $\T^\kappa$ is $\T$ without this left scattered maximum.

The Riemann--Stieltjes integral with respect to this measure is known as the delta integral
\be
\int_a^b f(t) \Delta t := \int_{[a,b)} f(t) d\sig(t), \qquad a,b\in\T.
\ee
There is also an alternate way \cite{ry} of defining the integral (and thus the measure) using
\be\label{delint}
\int_{[a,b)} f(t) d\sig(t) = \int_a^b f(\rho(t)) dt, \qquad a,b\in\T.
\ee
Indeed this equality is due to the fact that $\sigma$ is the image measure of the Lebesgue measure $\lambda$
under the function $\rho$, i.e.
\be\label{defimgmeas}
\sigma(A) = \lambda(\rho^{-1}(A))
\ee
for each Borel set $A$ (which is proved readily for intervals). 
Furthermore this shows that some measurable function $f$ is integrable with respect to $\sigma$ if and only
if $f\circ\rho$ is integrable with respect to Lebesgue measure.

A function $f$ on $\T$ is said to be delta (or Hilger) differentiable at some point $t\in\T$ if there is a 
number $f^\Delta(t)$ such that for every $\eps>0$ there is a neighborhood $U \subset\T$ of $t$ such that
\be
|f(\sig(t)) - f(s) - f^\Delta(t) (\sig(t)-s)| \le \eps |\sig(t)-s|, \quad s\in U.
\ee
If $\mu(t)=0$ then $f$ is differentiable at $t$ if and only if it is continuous at $t$ and
\be\label{hdermz}
f^\Delta(t) = \lim_{s\to t} \frac{f(s)-f(t)}{s-t}
\ee
exists (the limit has to be taken for $s\in\T\backslash \{ t\}$).
Similarly, if $\mu(t)>0$ then $f$ is differentiable at $t$ if and only if it is continuous at $t$ and
\be\label{hdermzn}
f^\Delta(t) = \frac{f(\sig(t)) - f(t)}{\mu(t)}
\ee
in this case.

Every function $f:\T \to \C$ can be extended to all of $\R$ via
\be
\bar{f}(t) =  f(\sig(t)), \qquad t\not\in\T.
\ee
Note that if the original function $f$ is continuous at $t\in\T$, then the extension will satisfy
$\bar{f}_-(t)=f(t)$ and $\bar{f}_+(t) = f(\sig(t))$. In particular, $\bar{f}$ will be left continuous
if $f$ is continuous.

Next we briefly review the concept of the derivative of a function on $\R$ with respect to the Borel measure $\sig$.
As already pointed out above, if $\nu:\R\to \C$ is locally of bounded variation we have an associated measure (denoted by the same letter for
notational simplicity) and we can consider the Radon--Nikodym derivative
\be
\frac{d\nu}{d\sig}(t)
\ee
which is defined a.e.\ with respect to $\sig$. We recall (see e.g., \cite[Sect.~1.6]{evga}) that
\be\label{eqnRADNIKderivative}
\frac{d\nu}{d\sig}(t) = \lim_{\eps\downarrow 0} \frac{\nu((t-\eps,t+\eps))}{\sig((t-\eps,t+\eps))} = 
\lim_{\eps\downarrow 0} \frac{\nu_-(t+\eps)-\nu_+(t-\eps)}{\sig_-(t+\eps)-\sig_+(t-\eps)},
\ee
where the limit exists a.e.\ with respect to $\sig$. 
The function $\nu$ is said to be absolutely continuous with respect to $\sig$ on some interval $[a,b)$ if the 
associated measure, restricted to this interval is absolutely continuous with respect to $\sig$, i.e. if
\be
\nu_-(x) - \nu_-(a) = \int_{[a,x)} \frac{d\nu}{d\sig}(t)d\sig(t), \quad x\in[a,b).
\ee
Furthermore $\nu$ is locally absolutely continuous with respect to $\sig$ if it is absolutely continuous on 
each such interval.
Note that in this case the only possible discontinuities of $\nu$ are the right scattered points.

\begin{lemma}\label{lem:hdrnd}
Suppose $f:\T\rightarrow\C$ is delta differentiable in some point $t\in\T^\kappa$ and $\bar{f}$ is locally of bounded variation.
Then the limit in~\eqref{eqnRADNIKderivative} exists and satisfies
\be
\frac{d\bar{f}}{d\sig}(t) = f^\Delta(t).
\ee
\end{lemma}

\begin{proof}
There are two possible cases:

(i) $\mu(t)=0$. First of all note that~\eqref{hdermz} implies
\[
\lim_{\eps\downarrow 0} \frac{\bar{f}_-(t+\eps)-f(t) - f^\Delta(t) \eta_+(\eps)}{\eta_+(\eps)} = 0, \qquad \eta_+(\eps)= \sig_-(t+\eps) -t.
\]
Indeed this follows because of $\bar{f}_-(t+\eps) = f(\sigma_-(t+\eps))$ and since
$\sigma_-(t+\eps)\rightarrow t$ as $\eps\downarrow0$ (also note that $\sigma_-(t+\eps)\in\T$).
Furthermore if $t$ is left dense we similarly obtain, using $\bar{f}_+(t-\eps) = f(\sigma_+(t-\eps))$
and $\sigma_+(t-\eps)\rightarrow t$ as $\eps\downarrow0$ (also note that $\sigma_+(t-\eps)\in\T$)
\[
\lim_{\eps\downarrow 0} \frac{f(t) - \bar{f}_+(t-\eps) - f^\Delta(t) \eta_-(\eps)}{\eta_-(\eps)} = 0, \qquad \eta_-(\eps)= t - \sig_+(t-\eps).
\]
Now observe that for each $\eps>0$ we have
\begin{align*}
\frac{\bar{f}_-(t+\eps)-\bar{f}_+(t-\eps)}{\sig_-(t+\eps) -\sig_+(t-\eps)} - f^\Delta(t)=& 
\frac{\bar{f}_-(t+\eps)-f(t) - f^\Delta(t) \eta_+(\eps)}{\eta_+(\eps)} \frac{\eta_+(\eps)}{\eta_+(\eps)+\eta_-(\eps)}\, +\\
& \frac{f(t) - \bar{f}_+(t-\eps) - f^\Delta(t) \eta_-(\eps)}{\eta_-(\eps)} \frac{\eta_-(\eps)}{\eta_+(\eps)+\eta_-(\eps)}.
\end{align*}
If $t$ is left scattered, then for small enough $\eps$ the second term vanishes, since then 
$f(t)=\bar{f}_+(t-\eps)$ as well as $\eta_-(\eps)=0$. 
Hence the claim follows since the first term converges to zero. 
Otherwise if $t$ is left dense, both terms converge to zero and the claim again follows (also note that the
fractions stay bounded since $\eta_+(\eps)$ and $\eta_-(\eps)$ are positive).

(ii) $\mu(t)>0$. Since $t$ is right scattered we have for small enough $\eps>0$
\[
\frac{\bar{f}_-(t+\eps)-\bar{f}_+(t-\eps)}{\sig_-(t+\eps) -\sig_+(t-\eps)} = \frac{f(\sig(t))-\bar{f}_+(t-\eps)}{\mu(t) +t -\sig(t-\eps)}
\to \frac{f(\sig(t))-f(t)}{\mu(t)}
\]
as $\eps\downarrow0$ and the claim follows from \eqref{hdermzn}.
\end{proof}

\noindent{\bf Example}: It might be interesting to note that the extension $\bar{f}$ need not be differentiable 
(in the usual sense) at some point $t\in\T$ if $f$ is delta differentiable at $t$ even not if $t$ is dense.
Indeed consider the time scale
\[
\T = \{ 0 \} \cup \left\lbrace \left. \pm\, t_n \,\right|\,n\in\N \right\rbrace, \qquad t_n=\frac{1}{n!},\quad n\in\N
\]
and the function
\[
f(0)=0, \qquad f\left(\pm\, t_n\right)=\frac{\pm 1}{(n+1)!}, \quad n\in\N.
\]
Then $f$ is delta differentiable at zero since
\[
f^\Delta(0) = \lim_{n\to\infty} \frac{f\left(\pm\, t_n\right) -f(0)}{\pm t_n} = 
 \lim_{n\to\infty} \frac{n!}{(n+1)!} = 0.
\]
However the extension $\bar{f}$ is not even right differentiable there, since
\[
\lim_{n\to\infty} \frac{\bar{f}\left(c\, t_n\right) -f(0)}{c\, t_n} = \lim_{n\to\infty} \frac{f\left(t_{n-1}\right) -f(0)}{c\, t_n}
 \lim_{n\to\infty} \frac{n!}{c\, n!} = \frac{1}{c} \ne f^\Delta(0),
\]
for each positive constant $c>1$.

As an immediate consequence of our lemma we obtain our main result:

\begin{theorem}
Suppose $f$ is delta differentiable for all $t\in\T^\kappa$ and $\bar{f}$ is locally of bounded variation. Then
the Radon--Nikodym derivative of $\bar{f}$ and the Hilger derivative of $f$ coincide at every point in $\T^\kappa$.
\end{theorem}

Concerning applications of this result we emphasize that it makes several results from measure theory
directly available to time scale calculus. For example, this result shows that the theory of generalized
differential equations with measure-valued coefficients as developed in the book by Mingarelli \cite{ming} contains
differential equation on time scales as a special case. We will use this in a follow up publication
\cite{etslts} to prove some new results about Sturm--Liouville equations on time scales based on some
recent extensions for Sturm--Liouville equations with measure-valued coefficients \cite{etslm}.

\section{Absolute continuity}
\label{sec:ac}

Absolutely continuous functions on time scales were introduced in~\cite{cabviv}.
Here we will denote them by delta absolutely continuous functions to distinguish them from
absolutely continuous functions in the usual measure theoretic definition.

Let $a$, $b\in\T$ with $a<b$ and $[a,b]_\T=[a,b]\cap\T$ be a subinterval of $\T$. A function $f:\T\rightarrow\C$ is 
said to be delta absolutely continuous on $[a,b]_\T$ if for every $\eps>0$, there exists a $\delta>0$ such that if
$\lbrace [a_k,b_k)\cap\T\rbrace_{k=1}^n$, with $a_k$, $b_k\in [a,b]_\T$ is a finite pairwise disjoint family of
subintervals of $[a,b]_\T$ with $\sum_{k=1}^n (b_k-a_k)<\delta$, then $\sum_{k=1}^n |f(b_k)-f(a_k)|<\eps$.

For functions which are delta absolutely continuous on $[a,b]_\T$, we have a variant of the fundamental theorem 
of calculus.

\begin{theorem}[{\cite[Theorem~4.1]{cabviv}}]\label{thmACFundTheo}
A function $f:\T\rightarrow\C$ is delta absolutely continuous on $[a,b]_\T$ if and only if $f$ is delta 
differentiable almost everywhere with respect to $\sigma$ on $[a,b)_\T$, $f^\Delta\in L^1([a,b)_\T;\sigma)$ and
\be
f(x) = f(a) + \int_{[a,x)_\T} f^\Delta(t) d\sigma(t), \quad x\in [a,b]_\T.
\ee
\end{theorem}

Note that if $f$ is delta absolutely continuous on $[a,b]_\T$, then the extension satisfies
\be
\bar{f}(x) = \bar{f}(a) + \int_{[a,x)} f^\Delta(t)d\sigma(t), \quad x\in[a,b].
\ee

The next lemma of Lebesgue is well known (e.g.\ Corollary~1 in Section~1.7.1 of \cite{evga} or Theorem~A.34 in \cite{tschroe}).

\begin{lemma}\label{lemLebp}
Let $g\in L^1(\R)$, then
\be
\lim_{\eps\downarrow0} \frac{1}{\eps} \int_{x-\eps}^{x+\eps} \left|g(t) - g(x)\right| dt = 0, 
\ee
for almost all $x\in\R$ with respect to Lebesgue measure.
\end{lemma}

%Some function $\bar{f}$ on $\R$ is said to be absolutely continuous with respect to $d\sigma$ on $[a,b]$ if it 
% is left-continuous on $[a,b]$, $\frac{d\bar{f}}{d\sigma}\in L^1([a,b];d\sigma)$ and
%\be
% \bar{f}(x) = \bar{f}(a) + \int_{[a,x)} \frac{d\bar{f}}{d\sigma}(t)d\sigma(t), \quad x\in[a,b].
%\ee

\begin{theorem}
Some function $f:\T\rightarrow\C$ is delta absolutely continuous on $[a,b]_\T$ if and only if $\bar{f}$ is 
left continuous on $[a,b]$ and absolutely continuous with respect to $\sigma$ on $[a,b)$. In this case 
\be
 f^\Delta(t) = \frac{d\bar{f}}{d\sigma}(t),
\ee
for almost all $t\in[a,b)$ with respect to $\sigma$.
\end{theorem}

\begin{proof}
If $f$ is delta absolutely continuous on $[a,b]_\T$ then $f$ is continuous on $[a,b]_\T$, hence $\bar{f}$ is
left continuous on $[a,b]$. Furthermore $\bar{f}$ is absolutely continuous with respect
to $\sigma$ on the interval $[a,b)$ by Theorem~\ref{thmACFundTheo}.
Conversely assume the extension $\bar{f}$ is left continuous on $[a,b]$ and absolutely continuous with respect to $\sigma$ on $[a,b)$, i.e. $\frac{d\bar{f}}{d\sigma}\in L^1([a,b);\sigma)$ and
\begin{align*}
\bar{f}(x) = \bar{f}(a) + \int_{[a,x)} \frac{d\bar{f}}{d\sigma}(t) d\sigma(t), \quad x\in[a,b].
\end{align*}
Then for each $t\in[a,b)_\T$ there are four cases:

(i) $t$ is an isolated point.
In this case $f$ is Hilger differentiable, with
\begin{align*}
f^\Delta(t) = \frac{\bar{f}(\sigma(t)) - \bar{f}(t)}{\sigma(t)-t}= \frac{\int_{[t,\sigma(t))} \frac{d\bar{f}}{d\sigma}(s)d\sigma(s)}{\sigma(t)-t} 
            = \frac{d\bar{f}}{d\rho}(t).
\end{align*}

(ii) $t$ is right scattered and left dense.
In this case for each small enough $\eps>0$ ($\sigma$ has no mass to the right of $t$) with $t-\eps\in\T$ we have
\begin{align*}
\left| \frac{f(\sigma(t)) - f(t-\eps)}{\sigma(t)-(t-\eps)} - \frac{d\bar{f}}{d\sigma}(t) \right| & = 
   \left|  \frac{\int_{[t-\eps,\sigma(t))} \frac{d\bar{f}}{d\sigma}(s) d\sigma(s)}{\sigma([t-\eps,t])} - \frac{\int_{[t-\eps,t]} \frac{d\bar{f}}{d\sigma}(t) d\sigma(s)}{\sigma([t-\eps,t])} \right| \\
 & \leq \frac{1}{\sigma([t-\eps,t])} \int_{[t-\eps,t]} \left|\frac{d\bar{f}}{d\sigma}(s) - \frac{d\bar{f}}{d\sigma}(t) \right| d\sigma(s).
\end{align*}
Now the right-hand side converges to zero as $\eps\downarrow0$, since the denominator is bounded from below by $\sigma(\lbrace t\rbrace)>0$.

(iii) $t$ is left scattered and right dense. These points are a null set with respect to $\sigma$.

(iv) $t$ is dense.
By redefining the Radon--Nikodym derivative on a null set we may assume that 
\begin{align*}
\frac{d\bar{f}}{d\sigma}(s) = \frac{d\bar{f}}{d\sigma}(\rho(s)), \quad s\not\in\T.
\end{align*}
From~\eqref{delint} we see that this function is integrable over $[a,b)$ with respect to the Lebesgue measure and that
\begin{align*}
\bar{f}(x) = \bar{f}(a) + \int_{[a,x)} \frac{d\bar{f}}{d\sigma}(s) d\sigma(s) = \bar{f}(a)+\int_a^x \frac{d\bar{f}}{d\sigma}(s) ds, \quad x\in[a,b]_\T.
\end{align*}
Now let $\eps>0$ with $t-\eps\in\T$ then
\begin{align*}
\left| \frac{f(\sigma(t)) - f(t-\eps)}{\sigma(t)-(t-\eps)} - \frac{d\bar{f}}{d\sigma}(t) \right| &
  = \left| \frac{1}{\eps} \int_{t-\eps}^t \frac{d\bar{f}}{d\sigma}(s)ds - \frac{1}{\eps}\int_{t-\eps}^t \frac{d\bar{f}}{d\sigma}(t) ds \right| \\
  & \leq \frac{1}{\eps} \int_{t-\eps}^t \left|\frac{d\bar{f}}{d\sigma}(s) - \frac{d\bar{f}}{d\sigma}(t) \right| ds \\
  & \leq \frac{1}{\eps} \int_{t-\eps}^{t+\eps} \left|\frac{d\bar{f}}{d\sigma}(s) - \frac{d\bar{f}}{d\sigma}(t) \right| ds.
\end{align*}
Similar one obtains for each $\eps>0$ with $t+\eps\in\T$ the estimate
\begin{align*}
\left| \frac{f(\sigma(t)) - f(t+\eps)}{\sigma(t)-(t+\eps)} - \frac{d\bar{f}}{d\sigma}(t) \right| \leq 
  \frac{1}{\eps} \int_{t-\eps}^{t+\eps} \left|\frac{d\bar{f}}{d\sigma}(s) - \frac{d\bar{f}}{d\sigma}(t)\right| ds.
\end{align*}
Now Lemma~\ref{lemLebp} shows that the Hilger derivative exists for 
almost all dense $t$ with respect to Lebesgue measure and coincides with the Radon--Nikodym derivative.
But since a Lebesgue null set of dense points is also a null set with respect to $\sig$, this and 
Theorem~\ref{thmACFundTheo} prove that $f$ is delta absolutely continuous on $[a,b]_\T$.
\end{proof}

Of course absolutely continuous functions have derivatives in the weak sense as introduced in \cite{agotpevi}. This
follows from the rule of integration by parts for functions of bounded variation \cite[Theorem~21.67]{hesr}.

\bigskip
\noindent
{\bf Acknowledgments.}
We thank Gusein Guseinov and Martin Bohner for helpful discussions and
hints with respect to the literature. We also thank one of the referees for pointing
out further references.


\begin{thebibliography}{XXXX}
\bibitem{agotpevi}
R. P. Agarwal, V. Otero-Espinar, K. Perera and D. R. Vivero, {\em Basic properties of Sobolev's spaces on time scales},
Adv. Difference Equ. {\bf 2006}, Article ID 38121, 1--14 (2006).
\bibitem{at}
F. Atkinson, {\em Discrete and Continuous Boundary Problems},
Academic Press, New York, 1964.
\bibitem{bope}
M. Bohner and A. Peterson, {\em Dynamic Equations on Time Scales},
Birkh\"auser, Boston, 2001.
\bibitem{bope2}
M. Bohner and A. Peterson (Eds.), {\em Advances in Dynamic Equations on Time Scales}, Birkh\"auser, Boston, 2003.
\bibitem{cabviv}
A. Cabada and D. R. Vivero, {\em Criterions for absolute continuity on time scales}, J. Difference Equ. Appl. {\bf 11}, 1013--1028 (2005).
\bibitem{cabviv2}
A. Cabada and D. R. Vivero, {\em Expression of the Lebesgue $\Delta$-integral on time scales as a usual Lebesgue integral;
application to the calculus of $Delta$-antiderivatives}, Math. Comput. Modelling {\bf 43} (2006), 194--207.
\bibitem{chfr}
C. J. Chyan and A. Fryszkowski, {\em Vitali lemma approach to differentiation on a time scale}, Studia Math. {\bf 162}, 161--173 (2004).
\bibitem{deuf} A. Deniz and \"U. Ufuktepe, {\em Lebesgue--Stieltjes measure on time scales},
Turkish J. Math. {\bf 33}, 27--40 (2009). 
\bibitem{dary}
F. A. Davidson and B. P. Rynne, {\em Self-adjoint boundary value problems on time scales}, Electron. J. Differential Equations {\bf 2007}, No. 175, 1--10 (2007).
\bibitem{etslm}
J. Eckhardt and G. Teschl, {\em Sturm--Liouville operators with measure-valued coefficients}, \arxiv{1105.3755}.
\bibitem{etslts}
J. Eckhardt and G. Teschl, {\em Sturm--Liouville operators on time scales}, J. Difference Equ. Appl. (to appear).
\bibitem{evga}
L. C. Evans and R. F. Gariepy, {\em Measure Theory and Fine Properties of Functions},
CRC Press, Boca Raton, 1992.
\bibitem{gus}
G. S. Guseinov, {\em Integration on time scales}, J. Math. Anal. Appl. {\bf 285}, 107--127 (2003).
\bibitem{hesr}
E. Hewitt and K. Stromberg, {\em Real and Abstract Analysis}, Springer, New York, 1965.
\bibitem{ming}
A. B. Mingarelli, {\em Volterra-Stieltjes Integral Equations and Generalized Ordinary Differential Expressions},
Lecture Notes in Math. {\bf 989}, Springer, Berlin, 1983.
\bibitem{mpt} D. Mozyrska, E. Paw{\l}uszewicz, and D.F.M. Torres, {\em The Riemann--Stieltjes integral on time scales}, 
Aust. J. Math. Anal. Appl. {\bf 7}, Art. 10, 14 pp (2010).
\bibitem{ry}
B. P. Rynne, {\em $L^2$ spaces and boundary value problems on time-scales}, J. Math. Anal. Appl. {\bf 328}, 1217--1236 (2007).
\bibitem{tschroe}
G. Teschl, {\em Mathematical Methods in Quantum Mechanics; With Applications to Schr\"odinger Operators},
Graduate Studies in Mathematics, Amer. Math. Soc., Rhode Island, 2009.

\end{thebibliography}
\end{document}